\documentclass[11pt,letterpaper,reqno]{amsart}

\usepackage{amsmath,amsthm,amssymb,thmtools,tikz-cd,comment,hyperref,cleveref,enumitem,breqn,float,caption,subfig,ytableau,tikz,cases,multirow,mathrsfs}
\usepackage{xcolor}
\usetikzlibrary{positioning, arrows.meta, calc}
\usepackage[T1]{fontenc}

\usetikzlibrary{decorations.pathreplacing,
	calligraphy,
	matrix}

\allowdisplaybreaks
\theoremstyle{plain}
\newtheorem{theorem}{Theorem}[section]
\newtheorem{lemma}[theorem]{Lemma}
\newtheorem{proposition}[theorem]{Proposition}

\newtheorem{thmalphabetintro}{Theorem}

\newtheorem{thmalphabetmaintext}{Theorem}

\theoremstyle{definition}

\newtheorem{example}[theorem]{Example}

\newtheorem{conjecture}[theorem]{Conjecture}
\newtheorem{question}[theorem]{Question}
\newtheorem{notation}[theorem]{Notation}

\newtheorem{open problem}[theorem]{Open Problem}
\newtheorem{remark and notation}[theorem]{Remark and Notation}
\newtheorem{remark and definition}[theorem]{Remark and Definition}
\newtheorem{definition and notation}[theorem]{Definition and Notation}
\newtheorem{notation and convention}[theorem]{Notation and convention}
\newtheorem{convention and notation}[theorem]{Convention and notation}

\def \p {\mathbb{P}}

\def \z {\mathbb{Z}}
\def \r {\mathbb{R}}
\def \c {\mathbb{C}}

\def \q {\mathbb{Q}}

\def \s {\mathbb{S}}

\def \Im {\operatorname{Im}}

\def \rank {\operatorname{rank}}

\def \det {\operatorname{det}}
\def \per {\operatorname{perm}}
\def \rk {\operatorname{\mathbf{R}}}
\def \wrk {\operatorname{\mathbf{WR}}}
\def \crk {\operatorname{\mathbf{CR}}}
\def \brk {\operatorname{\underline{\mathbf{R}}}}
\def \bwrk {\operatorname{\underline{\mathbf{WR}}}}
\def \bcrk {\operatorname{\underline{\mathbf{CR}}}}
\def \Id {\operatorname{Id}}
\def \sgn {\operatorname{sgn}}
\def \Hom {\operatorname{Hom}}
\def \End {\operatorname{End}}
\def \GL {\operatorname{GL}}
\def \SL {\operatorname{SL}}

\def \la {\langle}
\def \ra {\rangle}
\def \ld {\lambda}
\def \lt {\widetilde{\lambda}}
\def \Sym {\operatorname{Sym}}
\def \wt {\widetilde}
\def \Stab {\operatorname{Stab}}

\title[The border Waring rank of $x_1\cdots x_n$ is $2^{n-1}$]{The border Waring rank of $x_1\cdots x_n$ is $2^{n-1}$}
\author{Jong In Han}
\address{Jong In Han, School of Mathematics, Korea Institute for Advanced Study (KIAS), 85 Hoegi-ro, Dongdaemun-gu, Seoul, 02455, Republic of Korea}
\email{jihan09@kias.re.kr}


\begin{document}
\begin{abstract}
	In this paper, we show that the border Waring rank of $x_1x_2\cdots x_n$ over fields of characteristic zero is exactly $2^{n-1}$. As a consequence, the classical polarization identity is an optimal Waring decomposition even if we allow limits. As a symmetric tensor, this monomial is identified with the $n\times n$ permanent tensor. However, the lower bound is proved via the higher-order Koszul flattening of the $n\times n$ determinant tensor, which is not symmetric.
\end{abstract}
\maketitle

\section{Introduction}\label{sec:introduction}

The identity
\begin{equation}\label{eqn:identity}
	x_1x_2\cdots x_n=\frac{1}{2^{n-1}n!}\sum_{\substack{(\delta_1,\ldots,\delta_n)\in\{1,-1\}^{n}\\ \delta_1=1}}\left(\prod_{k=1}^n\delta_k\right)\left(\sum_{j=1}^n\delta_j x_j\right)^n
\end{equation}
is a classical polarization identity.
It follows from the polarization formula of Mazur and Orlicz \cite{zbMATH02540533} while the explicit form is given by Fischer \cite{MR1573008}.
It is also related to the classical theorem of Serret \cite{Ser1869} in 1869 (see \cite{MR3046335}).
It is also called Fischer's formula.

The \emph{Waring rank} $\wrk(f)$ of a homogeneous polynomial $f\in \Bbbk[x_1,\dots,x_n]_d$ over a field $\Bbbk$ is the smallest integer $r$ such that
\[
	f=a_1\ell_1^d+a_2\ell_2^d+\cdots+a_r\ell_r^d
\]
for some $\ell_1,\dots,\ell_r\in \Bbbk[x_1,\dots,x_n]_1$ and $a_1,\dots,a_r\in \Bbbk$.
Such a decomposition is called a \emph{Waring decomposition}.
Throughout this paper, we assume that the field $\Bbbk$ has characteristic zero.

An upper bound on the Waring rank can be obtained by giving a Waring decomposition.
Hence the polarization identity \eqref{eqn:identity} shows that
\[
	\wrk(x_1x_2\cdots x_n)\le 2^{n-1}.
\]
Ranestad and Schreyer \cite{MR2842085} proved $\wrk(x_1x_2\cdots x_n)\ge 2^{n-1}$, which determines the Waring rank as
\[
	\wrk(x_1x_2\cdots x_n)=2^{n-1}.
\]
Waring ranks of monomials over algebraically closed fields of characteristic zero are completely determined by Carlini, Catalisano, and Geramita \cite{MR2966824} as
\[
	\wrk(x_1^{\alpha_1}x_2^{\alpha_2}\cdots x_n^{\alpha_n})=\prod_{i=1}^{n-1} (\alpha_i+1)\quad\text{($\Bbbk=\overline{\Bbbk}$)}
\]
where $\alpha_1\ge\cdots\ge\alpha_n\ge 1$.
Determining Waring ranks of general monomials over non-algebraically-closed fields remains open.
For the study of Waring ranks of monomials over $\r$ and $\q$, see \cite{MR3648510,MR4461573}.

The \emph{border Waring rank} $\bwrk(f)$ of $f$ is the smallest integer $r$ such that $f$ lies in the Zariski closure of the set of $d$-forms of Waring rank at most $r$.
Since $\bwrk(f)\le \wrk(f)$, the polarization identity \eqref{eqn:identity} again gives
\[
	\bwrk(x_1x_2\cdots x_n)\le 2^{n-1}.
\]

It is natural to ask whether the polarization identity \eqref{eqn:identity} remains optimal in this limiting sense.
In contrast to the Waring rank, the border Waring rank of monomials is far from being determined (see \Cref{conj:monomial}).
Landsberg and Teitler \cite{MR2628829} proved
\[
	\binom{n}{\lfloor\frac{n}{2}\rfloor}\le\bwrk(x_1x_2\cdots x_n)\le 2^{n-1}.
\]
The exact value of $\bwrk(x_1x_2\cdots x_n)$ remained unknown after several attempts, and it appears as an open question in the recent survey \cite[Open question 3]{DL25}.

As a symmetric tensor, the monomial $x_1x_2\cdots x_n$ is identified to the $n\times n$ permanent tensor
\[
	\per_n=\sum_{\sigma\in\mathfrak{S}_n}e_{\sigma(1)}\otimes e_{\sigma(2)}\otimes\cdots\otimes e_{\sigma(n)}
\]
where $e_1,\dots,e_n$ denotes the standard basis of the vector space $V=\Bbbk^n$ and $\mathfrak{S}_n$ denotes the symmetric group.
The permanent and the determinant are a classical pair to compare in complexity theory.
In particular, separating the permanent and determinant \emph{polynomials in $n^2$ variables} is the content of Valiant's conjecture \cite{MR564634}.
In this paper, we consider instead $\per_n,\det_n\in V^{\otimes n}$ where
\[
	\det_n=\sum_{\sigma\in\mathfrak{S}_n}\sgn(\sigma)e_{\sigma(1)}\otimes e_{\sigma(2)}\otimes\cdots\otimes e_{\sigma(n)}.
\]
These two are separated by their tensor ranks for every $n\ge 3$ \cite{HJK25}.
However in this paper, the determinant is used as a tool to prove a statement about the permanent instead of being compared.
This converts a non-$\SL(V)$-equivariant problem to an $\SL(V)$-equivariant problem.

\begin{thmalphabetintro}\label{thmintro:main}
	It holds that $\bwrk(x_1x_2\cdots x_n)=2^{n-1}$ over fields of characteristic zero.
	Equivalently, $\bwrk(\per_n)=2^{n-1}$.
\end{thmalphabetintro}
As a consequence, the polarization identity \eqref{eqn:identity} is an optimal Waring decomposition even if we allow limits.

It is believed that the border Waring rank of the monomial $x_1^{\alpha_1}x_2^{\alpha_2}\cdots x_n^{\alpha_n}$ attains Landsberg--Teitler's upper bound $\prod_{i=2}^n (1+\alpha_i)$ (see \cite[Section 6.1]{MR4332674}).
This conjecture had already circulated in the community for almost a decade by 2019 \cite{CGO19}.

\begin{conjecture}[{\cite[Conjecture 1.1]{Oed16}}]\label{conj:monomial}
For $\alpha_1\ge\alpha_2\ge\cdots\ge \alpha_n\ge 1$, it holds that
\[
	\bwrk(x_1^{\alpha_1}x_2^{\alpha_2}\cdots x_n^{\alpha_n})=\prod_{i=2}^n (1+\alpha_i).
\]
\end{conjecture}

The conjecture was proved for monomials with three variables by Buczy{\'n}ska and Buczy{\'n}ski \cite{MR4332674}, and for monomials satisfying
\begin{equation}\label{eqn:MV}
	\alpha_1\ge \alpha_2+\alpha_3+\cdots+\alpha_n-2
\end{equation}
by Ma{\'n}dziuk and Ventura \cite{MR5014850}.
For the squarefree monomial $x_1x_2\cdots x_n$, the condition \eqref{eqn:MV} becomes $n\le 4$.

Previously, the border Waring rank of $x_1x_2\cdots x_n$ was determined up to $n=5$.
The proof of the case $n=3$ can be found in Ottaviani's study \cite{MR2502909}:
\[
	\bwrk(x_1x_2x_3)=4.
\]
Guan \cite{MR3627645} proved $\bwrk(x_1\cdots x_4)\ge 7$, $\bwrk(x_1\cdots x_5)\ge 14$, and
\[
	\bwrk(x_1x_2\cdots x_{2k+1})\ge \binom{2k+1}{k}\left(1+\frac{k^2}{(k+1)^2	(2k-1)}\right),
\]
which gives the previous best known bounds for odd $n\ge 9$.
The exact border Waring ranks of $x_1\cdots x_4$ and $x_1\cdots x_5$ are determined by Buczy{\'n}ska and Buczy{\'n}ski \cite[Example 6.6]{MR4332674} as
\[
	\bwrk(x_1\cdots x_4)=8\text{ and }\bwrk(x_1\cdots x_5)=16.
\]
It is known that $\brk(x_1\cdots x_6)\ge 29$ and $\brk(x_1\cdots x_7)\ge 55$ \cite{HJK25}.
As a consequence, we get
\[
	\bwrk(x_1\cdots x_6)\ge 29\text{ and }\bwrk(x_1\cdots x_7)\ge 55.
\]
It is also known that $\brk(x_1\cdots x_n)\ge \binom{n}{\lfloor\frac{n}{2}\rfloor}+1$ over $\c$ and its subfields \cite{HS26}.
Consequently, it holds that
\[
	\bwrk(x_1\cdots x_n)\ge \binom{n}{\lfloor\frac{n}{2}\rfloor}+1.
\]
The following table summarizes these previous results.

\begin{table}[H]
	\resizebox{\textwidth}{!}{%
	\begin{tabular}{|c|cccccccccc|}
		\hline
		$n$ & $3$ & $4$ & $5$ & $6$ & $7$ & $8$ & $9$ & $10$ & $11$ & $\cdots$ \\
		\hline
		upper bound & $4$ & $8$ & $16$ & $32$ & $64$ & $128$ & $256$ & $512$ & $1024$ & $\cdots$ \\
		lower bound & $4$ & $8$ & $16$ & $29$ & $55$ & $71$ & $138$ & $253$ & $498$ & $\cdots$\\
		reference for lower bound& \cite{MR2502909} & \cite{MR4332674} & \cite{MR4332674} & \cite{HJK25} & \cite{HJK25} & \cite{HS26} & \cite{MR3627645} & \cite{HS26} & \cite{MR3627645} & $\cdots$ \\ 
		\hline
	\end{tabular}
	}%
	\caption{Previous best known bounds on $\bwrk(x_1x_2\cdots x_n)$ over $\c$ and its subfields}%
	\label{table:previous}%
\end{table}%

\section*{Acknowledgement}
The author thanks Jaros{\l}aw Buczy{\'n}ski, Fulvio Gesmundo, Kangjin Han, J. M. Landsberg, Luke Oeding, and Alessandro Oneto for their interest in this work.
The author was supported by a KIAS Individual Grant (MG101401) at Korea Institute for Advanced Study.

\section{Representation theory and applications to ranks}\label{sec:representation_theory}

In this section, we recall some basic facts in representation theory and methods giving lower bounds for ranks.
We generally follow \cite{MR1153249} for the representation theory while some conventions might be different.

We suppose $V$ is a vector space of dimension $n$ over an algebraically closed field $\Bbbk$ of characteristic zero with the basis $\{e_1,\dots,e_n\}$.
Although our main result is stated for arbitrary fields of characteristic zero, it is enough to prove it for algebraically closed fields.
We denote by $\Sym^d V$ the space of the homogeneous polynomials of degree $d$ in variables $x_1,\dots,x_n$ where $x_i:=e_i$.
We also identify $\Sym^d V\cong\s_{(d)}V$ via $v_1\cdots v_d\mapsto (v_1\otimes \cdots\otimes v_d)\cdot c_{(d)}$ and $\bigwedge^d V\cong\s_{(1^d)}V$ via $v_1\wedge\cdots\wedge v_d\mapsto (v_1\otimes \cdots\otimes v_d)\cdot c_{(1^d)}$, defined in \Cref{subsec:Schur_module}.
Here, the $v_i$'s are the elements of $V$.

\subsection{Young diagram and Young tableau}
A \emph{partition} $\ld$ of a positive integer $d$ is a sequence $(\ld_1,\dots,\ld_k)$ of positive integers such that $\ld_1\ge \cdots\ge \ld_k$ and $\sum_{i=1}^k \ld_i=d$.
We denote by $\ld\vdash d$ when $\ld$ is a partition of $d$.
The \emph{length} $\ell(\ld)$ of $\ld$ is defined to be $k$ in this case.
The partition with $\lambda_1=\cdots=\lambda_k=a$ is denoted by $(a^k)$.

Given a partition $\ld=(\ld_1,\dots,\ld_k)$, the \emph{Young diagram} associated to $\ld$ is the diagram consisting of left-aligned boxes such that the $i$-th row has $\ld_i$ boxes.
We also write $\ld$ for the associated Young diagram and often identify these two.
The \emph{conjugate partition} $\ld'$ to $\ld$ is the partition corresponding to the transpose of the Young diagram associated to $\ld$.

For a Young diagram $\ld$ and a positive integer $m$, a \emph{Young tableau} on $\ld$ with entries in $\{1,\dots,m\}$ is an assignment of a number in $\{1,\dots,m\}$ to each box of $\ld$.
In this case, we say that such a Young tableau $T$ has the \emph{shape} $\ld$.
A \emph{semistandard Young tableau} (SSYT) is a Young tableau such that the numbers in each row are nondecreasing and the numbers in each column are strictly increasing.

\subsection{Young symmetrizer and Schur module}\label{subsec:Schur_module}
Suppose $\ld\vdash d$ and consider a Young tableau $T$ of shape $\ld$ in which each of $1,\dots,d$ appears exactly once.
We define two subgroups of the symmetric group $\mathfrak{S}_d$ as follows.
\begin{align*}
	P_T &= \{\sigma\in\mathfrak{S}_d\mid \sigma\text{ preserves the set of numbers in each row}\} \\
	Q_T &= \{\sigma\in\mathfrak{S}_d\mid \sigma\text{ preserves the set of numbers in each column}\}.
\end{align*}
The \emph{group algebra} $\Bbbk\mathfrak{S}_d$ is a $\Bbbk$-algebra whose underlying vector space has basis $\{e_\sigma\mid \sigma\in\mathfrak{S}_d\}$.
Its algebra structure is given by $e_\sigma\cdot e_\tau = e_{\sigma\tau}$.
Then we get the induced elements in $\Bbbk\mathfrak{S}_d$ as follows.
\[
	a_T := \sum_{\sigma\in P_T} e_\sigma\text{ and }b_T := \sum_{\sigma\in Q_T} \sgn(\sigma)e_\sigma
\]

The \emph{Young symmetrizer} $c_T$ associated to $T$ is defined as $c_T = a_T\cdot b_T\in \Bbbk\mathfrak{S}_d$.
The right action $(v_1\otimes v_2\otimes \cdots\otimes v_d)\cdot \sigma := v_{\sigma(1)}\otimes v_{\sigma(2)}\otimes \cdots\otimes v_{\sigma(d)}$ of $\mathfrak{S}_d$ on $V^{\otimes d}$ induces a right action of $\Bbbk\mathfrak{S}_d$ on $V^{\otimes d}$.
The \emph{Schur module} associated to $T$ is defined by $\s_TV=\left(V^{\otimes d}\right)\cdot c_T$.
Let $T_\ld$ be the Young tableau of shape $\ld$ filled with $1,\dots,d$ in the reading order as in the following example.
\[
	T_\ld=\ytableaushort{1234,567,89}
\]
Then we define $a_\ld:=a_{T_\ld}$, $b_\ld:=b_{T_\ld}$, $c_\ld:=c_{T_\ld}$, and $\s_\ld V:=(V^{\otimes d})\cdot c_\ld=\s_{T_\ld} V$.

The Schur module $\s_\ld$ is an irreducible $\GL(V)$-representation whenever $\ell(\ld)\le \dim V$.
As $(V^{\otimes d})\cdot c_\ld=\left((V^{\otimes d})\cdot a_\ld\right)\cdot b_\ld\subseteq\Im(b_\ld)$, the Schur module $\s_\ld V$ can be considered as a subrepresentation of	$W_{\ld}:=\bigwedge^{\ld'_1}V\otimes \bigwedge^{\ld'_2}V\otimes\cdots\otimes \bigwedge^{\ld'_{k'}}V$ by identifying $(v_1\otimes\cdots\otimes v_d)\cdot b_\ld$ with the tensor product of wedge products of the $v_i$'s corresponding to each column.

\begin{notation}\label{not:e_T}
	For a given Young tableau $T$ of shape $\ld=(\ld_1,\dots,\ld_k)$ with entries in $\{1,\dots,n\}$, we denote
\[
	e_T:=(e_{T_{1,1}}\otimes e_{T_{1,2}}\otimes\cdots\otimes e_{T_{1,\ld_1}})\otimes (e_{T_{2,1}}\otimes \cdots\otimes e_{T_{2,\ld_2}})\otimes\cdots\otimes (e_{T_{k,1}}\otimes\cdots\otimes e_{T_{k,\ld_k}})
\]
where $T_{i,j}$ is the entry of $T$ in the $i$-th row and $j$-th column.
\end{notation}

\begin{example}
	For the Young tableau
\[
	T=\ytableaushort{1336,24},
\]
we get $e_T=e_1\otimes e_3\otimes e_3\otimes e_6\otimes e_2\otimes e_4\in V^{\otimes 6}$.
\end{example}
As $T$ runs over Young tableaux of shape $\ld=(\ld_1,\dots,\ld_k)$ with entries in $\{1,\dots,n\}$, the elements $e_T$ form the standard basis of $V^{\otimes d}$.
The Schur module $\s_\ld V$ has the basis \[\{e_S\cdot c_\ld\mid S\text{ is an SSYT of shape }\ld\text{ with entries in }\{1,\dots,n\}\}.\]
Hence the dimension of $\s_\ld V$ equals the number of SSYTs of shape $\ld$ with entries in $\{1,2,\dots,n\}$.
It also holds that
\[
	\dim \s_\ld V=\prod_{1\le i<j\le n}\frac{\ld_i-\ld_j+j-i}{j-i}
\]
where we regard $\ld_i=0$ for $i>\ell(\ld)$.

\subsection{Pieri map and Young flattening}\label{subsec:Young_flattening}

Schur's lemma is one of the tools that we use in the proof repeatedly.
\begin{lemma}[Schur's lemma]
	Let $U$ and $W$ be irreducible $G$-representations.
	Then a $G$-equivariant map $\Phi:U\to W$ is either an isomorphism or the zero map.
	Furthermore, if $W=U$, then $\Phi$ is a scalar multiple of $\Id_U$ where $\Id_U:U\to U$ is the identity map.
\end{lemma}

Pieri's formula is a $\GL(V)$-equivariant isomorphism in the below.
\[
	\s_{(d)}V\otimes \s_\ld V \cong \bigoplus_{\mu} \s_\mu V
\]
Here, the direct sum is taken over all Young diagrams $\mu$ obtained by adding $d$ boxes to $\ld$, with no two in the same column.
Since there is the unique way to add boxes to obtain such a shape $\mu$, the multiplicity of each $\s_\mu V$ on the right hand side is one.
By composing with the projection to $\s_\mu V$, we get the Pieri map \[\s_{(d)}V\otimes \s_\ld V \to\s_\mu V\] which is $\GL(V)$-equivariant and unique up to scalar by Schur's lemma.
We denote the induced $\GL(V)$-map by
\[
	\mathcal{F}_{\ld,\mu}:\s_{(d)}V\to\Hom_\Bbbk(\s_\ld V,\s_\mu V).
\]
Recall that we identify $\s_{(d)}V=\Sym^dV$. Hence for a homogeneous polynomial $f\in \Sym^dV$, we get the linear map
\[
	\mathcal{F}_{\ld,\mu}(f):\s_\ld V \to\s_\mu V.
\]
Note that $\mathcal{F}_{\ld,\mu}(f)$ need not be $\GL(V)$-equivariant or $\SL(V)$-equivariant while $\mathcal{F}_{\ld,\mu}$ is.

Landsberg and Ottaviani \cite{MR3081636} introduced a rank method called \emph{Young flattening}.
\begin{theorem}[{\cite[Proposition 4.1.1]{MR3081636}}]\label{thm:Young_flat}
	Under the above notations, it holds that
	\[
		\bwrk(f)\ge \dfrac{\rank\mathcal{F}_{\ld,\mu}(f)}{\rank\mathcal{F}_{\ld,\mu}(x_1^d)}.
	\]
\end{theorem}

\subsection{Higher-order Koszul flattening}
The \emph{tensor rank} $\rk(T)$ of a tensor $T\in V_1\otimes\cdots\otimes V_n$ is the smallest integer $r$ such that $T$ can be written as a sum of $r$ simple tensors.
The \emph{border tensor rank} $\brk(T)$ is the smallest integer $r$ such that $T$ lies in the Zariski closure of the set of tensors of rank at most $r$.
For a symmetric tensor $f\in \Sym^d V$, it holds that $\rk(f)\le \wrk(f)$ and $\brk(f)\le \bwrk(f)$.

Hauenstein, Oeding, Ottaviani, and Sommese \cite{MR3987862} proposed to apply the Koszul flattening, introduced by Landsberg and Ottaviani \cite{MR3376667}, simultaneously to higher-order tensors.
\begin{theorem}[{\cite[Theorem 2.1]{MR3376667},\cite[Section 5.1]{MR3987862}}]\label{thm:higher_Koszul}
	Let
	\[
		T=\sum_{i=1}^kv_{1,i}\otimes\cdots\otimes v_{n,i}\in V_1\otimes\cdots\otimes V_n
	\]
	and $\dim V_i=d_i$.
	Fix nonnegative integers $p_1,\dots,p_n$.
	Define
	\[
		T^{\wedge(p_1,\dots,p_n)}_{(V_1,\dots,V_n)}:\bigwedge^{p_1}V_1\otimes\bigwedge^{p_2}V_2\otimes\cdots\otimes\bigwedge^{p_n}V_n\to\bigwedge^{p_1+1}V_1\otimes\bigwedge^{p_2+1}V_2\otimes\cdots\otimes\bigwedge^{p_n+1}V_n
	\]
	by
	{\scriptsize
	\begin{align*}
		&(a_{1,1}\wedge\cdots\wedge a_{1,p_1})\otimes (a_{2,1}\wedge\cdots\wedge a_{2,p_2})\otimes\cdots\otimes (a_{n,1}\wedge\cdots\wedge a_{n,p_n})\\
		&\mapsto\sum_{i=1}^{k}(v_{1,i}\wedge a_{1,1}\wedge\cdots\wedge a_{1,p_1})\otimes (v_{2,i}\wedge a_{2,1}\wedge\cdots\wedge a_{2,p_2})\otimes\cdots\otimes (v_{n,i}\wedge a_{n,1}\wedge\cdots\wedge a_{n,p_n})
	\end{align*}
	}%
	and extending linearly.
	Then
	\[
		\brk(T)\ge \frac{\rank\left(T^{\wedge(p_1,\dots,p_n)}_{(V_1,\dots,V_n)}\right)}{\binom{d_1-1}{p_1}\cdots\binom{d_n-1}{p_n}}.
	\]
\end{theorem}
Note that both $(\det_n)^{\wedge(0,1,\dots,n-2,n-1)}$ and $(\per_n)^{\wedge(0,1,\dots,n-2,n-1)}$ correspond to square matrices.
It is known that $(\det_n)^{\wedge(0,1,\dots,n-2,n-1)}$ has full rank \cite{HJK25}.
However, $(\per_n)^{\wedge(0,1,\dots,n-2,n-1)}$ need not have full rank in general.

\section{Border Waring ranks of monomials}\label{sec:bwrk}

\subsection{Landsberg--Teitler's bounds}
For $\alpha=(\alpha_1,\alpha_2,\dots,\alpha_n)$ and $\delta\ge 0$, denote by $S_{\alpha,\delta}$ the number of distinct $n$-tuples $(\beta_1,\dots,\beta_n)$ such that $\sum_{i=1}^n \beta_i=\delta$ and $0\le\beta_i\le\alpha_i$.
Landsberg and Teitler established bounds on border Waring ranks of monomials.
\begin{theorem}[{\cite[Theorem 11.2]{MR2628829}}]\label{thm:LTbound}
	Let $\alpha_1\ge\alpha_2\ge\cdots\ge\alpha_n\ge 1$ and $d=\sum_{i=1}^n \alpha_i$.
	Then it holds that
	\[
		S_{\alpha,\lfloor \frac{d}{2}\rfloor}\le\bwrk(x_1^{\alpha_1}x_2^{\alpha_2}\cdots x_n^{\alpha_n})\le \prod_{i=2}^n (1+\alpha_i).
	\]
\end{theorem}

In particular, it gives
\[
	\binom{n}{\lfloor\frac{n}{2}\rfloor}\le \bwrk(x_1x_2\cdots x_n)\le 2^{n-1}.
\]

\subsection{Proof of the lower bound}
In this section, we prove the main theorem.
Note that as the Young flattening of $x_1\cdots x_n$ is neither $\GL(V)$-equivariant nor $\SL(V)$-equivariant, we cannot use Schur's lemma in the usual manner.
It is still true that
\[
	(\Bbbk^{*})^{n-1}\rtimes\mathfrak{S}_n\subseteq \Stab_{\GL(V)}(\per_n)
\]
where $(\Bbbk^*)^{n-1}$ acts by $e_i\mapsto a_ie_i$ with $a_1\cdots a_n=1$ and $\mathfrak{S}_n$ acts by permuting the $e_i$'s.
However, the domain and codomain of the Young flattening of $\per_n$ need not be irreducible as $\left((\Bbbk^*)^{n-1}\rtimes \mathfrak{S}_{n}\right)$-representations, and the map is hard to track as such representations.
Instead, we perform the essential part of the argument in the $\SL(V)$-equivariant setting.

\begin{thmalphabetmaintext}\label{thm:main}
	It holds that $\bwrk(x_1x_2\cdots x_n)=2^{n-1}$ over fields of characteristic zero.
	Equivalently, $\bwrk(\per_n)=2^{n-1}$.
\end{thmalphabetmaintext}
\begin{proof}
It is enough to prove that $\bwrk(x_1x_2\cdots x_n)\ge 2^{n-1}$ when $\Bbbk$ is algebraically closed.
As in \Cref{sec:representation_theory}, we identify $x_i=e_i\in V$.
We pick the partitions $\ld=(n-1,n-2,\dots,1)$, $\lt=(n,n-1,\dots,1)$, and regard these as Young diagrams.\footnote{These shapes were suggested by Oeding \cite{Oed16}. For details, see \Cref{subsec:Oeding}.}
The higher-order Koszul flattening of the $n\times n$ determinant tensor $\det_n$ with $(p_1,p_2,\dots,p_n)=(0,1,\dots,n-1)$ is
\begin{align*}
	(\det_n)^{\wedge(0,1,\dots,n-1)}_{(V_1,V_2,\dots,V_{n})}:&\bigwedge^{n-1} V_{n}\otimes\bigwedge^{n-2}V_{n-1}\otimes\cdots\otimes\bigwedge^1 V_2\\*
	&\to \bigwedge^n V_{n}\otimes \bigwedge^{n-1} V_{n-1}\otimes \cdots \otimes \bigwedge^2 V_2\otimes \bigwedge^1 V_1.
\end{align*}
We simply denote this map by $(\det_n)^{\wedge}$.
As $\det_n$ is $\SL(V)$-invariant and the higher-order Koszul flattening map $z\mapsto (z)^{\wedge}$ is $\GL(V)$-equivariant, the map $(\det_n)^{\wedge}$ is $\SL(V)$-equivariant.

Consider the restriction
\[
	(\det_n)^{\wedge}|_{\s_\ld V}:\s_\ld V\to  \bigwedge^n V\otimes \bigwedge^{n-1} V\otimes \cdots \otimes \bigwedge^2 V\otimes V,
\]
and denote $W_{\lt}:=\bigwedge^n V\otimes \bigwedge^{n-1} V\otimes \cdots \otimes \bigwedge^2 V\otimes V$.
We decompose $W_{\lt}$ into irreducible $\SL(V)$-subrepresentations:
\[
	W_{\lt} = \bigoplus_\tau\s_\tau V.
\]

We claim that $\s_{\lt} V$ has multiplicity one on the right hand side.
From Pieri's formula, we have
\[
	\s_\tau V\otimes\bigwedge^m V\cong \bigoplus_{\nu\in \operatorname{ext}(\tau)}\s_\nu V 
\]
where $\operatorname{ext}(\tau)$ denotes the list of Young diagrams that are obtained by adding $m$ boxes to $\tau$, with no two in the same row.
Suppose that we have a representation \[U:=\bigoplus_{\pi\in L}\s_\pi V,\] where $L$ is a collection of some Young diagrams $\pi$ such that $\pi\vdash \sum_{i=1}^{m-1} i$ and $(m-1,m-2,\dots,1)$ appears in $L$ with multiplicity one.
Then
\begin{align}
	U\otimes\bigwedge^mV &\cong \left(\bigoplus_{\pi\in L}\s_\pi V\right)\otimes \bigwedge^m V\notag\\
	&\cong \bigoplus_{\pi\in L}\left(\s_\pi V\otimes \bigwedge^m V\right)\notag\\
	&\cong \bigoplus_{\pi\in L}\left(\bigoplus_{\nu\in\operatorname{ext}(\pi)}\s_\nu V \right).\label{eqn:induction_box}
\end{align}
The Young diagram $(m,m-1,m-2,\dots,1)$ appears only once as a summand in \eqref{eqn:induction_box} due to Pieri formula's box adding rule.
Indeed, $\tau=(m-1,m-2,\dots,1)$ is the only element in $L$ such that $\operatorname{ext}(\tau)$ contains $(m,m-1,m-2,\dots,1)$.
Therefore the claim holds by induction on $m$.

Note that $\lt$ is the only shape among the $\tau$'s on the right hand side of $W_{\lt} = \bigoplus_\tau\s_\tau V$ such that $\s_\tau V\cong \s_\ld V$ as $\SL(V)$-representations.
Indeed, $\s_\tau V\cong \s_\ld V$ implies that there is an integer $a$ such that $\tau=\ld+(a^n)$.
As $\tau\vdash\sum_{i=1}^ni$ and $\ld\vdash\sum_{i=1}^{n-1}i$, we get $a=1$.
In particular, $\s_\lt V$ itself is the isotypic component of $W_\lt$ as it has multiplicity one also.

Denote the highest weight vectors of $\s_\ld V$ and $\s_\lt V$ by
\begin{align*}
	h &:=(e_1\wedge e_2\wedge\cdots \wedge e_{n-1})\otimes(e_1\wedge \cdots \wedge e_{n-2})\otimes\cdots\otimes e_1\in\s_\ld V\\
	\wt{h} &:=(e_1\wedge e_2\wedge\cdots \wedge e_n)\otimes(e_1\wedge \cdots \wedge e_{n-1})\otimes\cdots\otimes e_1\in\s_\lt V.
\end{align*}
Then we have
\[
	(\det_n)^{\wedge}|_{\s_\ld V}(h)=(-1)^{\binom{n}{2}}\wt{h}.
\]
Therefore, $(\det_n)^{\wedge}|_{\s_\ld V}(h)\in \s_{\lt} V$.
Then by Schur's lemma, we may denote the map as
\[
	(\det_n)^{\wedge}|_{\s_\ld V}: \s_{\ld} V\to  \s_{\lt} V
\]
since $\s_\lt V$ is an isotypic component.
Consider the map
\begin{align*}
	\mathscr{C}_\lt:V^{\otimes \binom{n+1}{2}} &\to\s_\lt V\\
	v&\mapsto v\cdot c_\lt
\end{align*}
where we regard $\s_\lt V\subseteq W_\lt\subseteq V^{\otimes \binom{n+1}{2}}$ as in \Cref{subsec:Schur_module}.
Note that the map
\begin{align*}
	\Phi: \s_\ld V&\to W_\lt\\
	v&\mapsto (e_1\wedge e_2\wedge \cdots\wedge e_n)\otimes v
\end{align*}
is $\SL(V)$-equivariant since $e_1\wedge e_2\wedge \cdots\wedge e_n=\det_n$ is $\SL(V)$-invariant.
Furthermore, $\Phi(h)=\widetilde{h}\in \s_\lt V$.
As $\s_\lt V$ is an isotypic component of $W_\lt$, we get $\Im(\Phi)\subseteq\s_\lt V$.
Hence we regard this map as
\begin{align*}
	\Phi: \s_\ld V&\to \s_\lt V
\end{align*}
which must be an isomorphism by Schur's lemma.

Define
\begin{align*}
	\Psi: V^{\otimes n} &\to \End_\Bbbk(\s_\ld V)\\*
	z &\mapsto \Phi^{-1}\circ \mathscr{C}_\lt\circ (z)^{\wedge}|_{\s_\ld V}.
\end{align*}
Then we consider
\[
	\Psi(\det_n)=\Phi^{-1}\circ \mathscr{C}_\lt\circ (\det_n)^{\wedge}|_{\s_\ld V}:\s_\ld V \to \s_\ld V.
\]
Following \Cref{not:e_T}, denote $e_{T_0}=(e_1)^{\otimes n}\otimes (e_2)^{\otimes n-1}\otimes \cdots\otimes (e_n)\in V^{\otimes\binom{n+1}{2}}$ where $T_0$ is a Young tableau of shape $\lt$ filled with the number $i$ at all boxes in the $i$-th row.

Then we get
\begin{align*}
	\Psi(\det_n)(h) &= \Phi^{-1}\circ \mathscr{C}_\lt\circ(\det_n)^{\wedge}|_{\s_\ld V}(h) \\
	&= \Phi^{-1}\circ \mathscr{C}_\lt \left((-1)^{\binom{n}{2}}\wt{h}\right) \\
	&= \Phi^{-1}\circ \mathscr{C}_\lt\left((-1)^{\binom{n}{2}}\frac{1}{\prod_{i=1}^n i!}e_{T_0}\cdot c_\lt\right)\\
	&= \Phi^{-1}\left((-1)^{\binom{n}{2}}\frac{1}{\prod_{i=1}^n i!}e_{T_0}\cdot c_\lt^2\right)\\
	&= \Phi^{-1}\left((-1)^{\binom{n}{2}}\frac{1}{\prod_{i=1}^n i!}e_{T_0}\cdot n_\lt c_\lt\right)\\
	&= \Phi^{-1}\left((-1)^{\binom{n}{2}}n_\lt \wt{h}\right)\\
	&= (-1)^{\binom{n}{2}}n_\lt h.
\end{align*}
Here, $n_\lt\in\Bbbk$ is a scalar such that $c_\lt^2=n_\lt c_\lt$.
It is well known that $n_\lt$ is the product of hook lengths on $\lt$ (cf. \cite[Formula 4.12 and Lemma 4.26]{MR1153249}).
From the shape of $\lt$, we get
\[
	n_\lt = \prod_{i=1}^n\left(2i-1\right)^{n+1-i},
\]
hence
\[
	\Psi(\det_n)(h) = (-1)^{\binom{n}{2}}\prod_{i=1}^n\left(2i-1\right)^{n+1-i}h.
\]
Then, by Schur's lemma,
\[
	\Psi(\det_n) = \left((-1)^{\binom{n}{2}}\prod_{i=1}^n\left(2i-1\right)^{n+1-i}\right)\Id_{\s_\ld V}.
\]
Note that
\[
	\dim \s_\ld V=\prod_{1\le i<j\le n}\frac{\ld_i-\ld_j+j-i}{j-i}=2^{\binom{n}{2}}.
\]
Thus, whenever we fix a basis $\mathcal{B}$ on both the source and the target, we get a matrix $M_{\det_n}$ corresponding to $\Psi(\det_n)$ such that
\[
	\det(M_{\det_n})=\left((-1)^{\binom{n}{2}}\prod_{i=1}^n\left(2i-1\right)^{n+1-i}\right)^{2^{\binom{n}{2}}}.
\]
Note that this does not depend on the choice of the basis $\mathcal{B}$.
Moreover, as $\Psi(\det_n)$ is an integer multiple of $\Id_{\s_\ld V}$, the matrix $M_{\det_n}$ has integer entries for every choice of $\mathcal{B}$.
We will fix the basis later.

Finally, we drop the $\SL(V)$-equivariance.
Let $\per_n=x_1\cdots x_n$ and consider the map
\[
	\Psi(\per_n)=\Phi^{-1}\circ \mathscr{C}_\lt\circ (\per_n)^{\wedge}|_{\s_\ld V}:\s_\ld V\to\s_{\ld} V.
\]
As pointed out in \Cref{subsec:Young_flattening}, any $\GL(V)$-equivariant map $\Sym^nV\otimes \s_\ld V\to\s_\lt V$ is a scalar multiple of the Pieri map $\Sym^nV\otimes \s_\ld V\to\s_\lt V$ by Schur's lemma.
Since the higher-order Koszul flattening map $z\mapsto (z)^{\wedge}$ and $\mathscr{C}_\lt$ are $\GL(V)$-equivariant, the map
\begin{align*}
	\mathscr{C}_\lt\circ (-)^{\wedge}|_{\s_\ld V}:\Sym^n V\otimes \s_\ld V &\to \s_\lt V\\*
	f\otimes v &\mapsto \mathscr{C}_\lt\circ (f)^{\wedge}|_{\s_\ld V}(v)
\end{align*}
is $\GL(V)$-equivariant, and we get
\begin{equation}\label{eqn:scalar}
	\mathscr{C}_\lt\circ (-)^{\wedge}|_{\s_\ld V}=\gamma\mathcal{F}_{\ld,\lt}(-)
\end{equation}
for some $\gamma\in\Bbbk$ where $\mathcal{F}_{\ld,\lt}(-)$ denotes the Young flattening.
Note that only the left hand side of \eqref{eqn:scalar} extends to arbitrary tensors in $V^{\otimes n}$.

We claim that there exists a basis $\mathcal{B}$ of $\s_\ld V$ such that for every
\[
	z\in \z\la e_{i_1}\otimes \cdots\otimes e_{i_n}\mid i_1,\dots,i_n\in \{1,\dots,n\}\ra\subseteq V^{\otimes n},
\]
the matrix $M_z$ corresponding to $\Psi(z)=\Phi^{-1}\circ \mathscr{C}_\lt\circ (z)^{\wedge}|_{\s_\ld V}:\s_\ld V\to\s_{\ld}V$ with respect to the basis $\mathcal{B}$ has integer entries.
Let
\[
	A:=\z\left\la e_T\cdot c_\ld\ \middle|\ 
	\begin{aligned}
		&T\text{ is a Young tableau of shape }\ld\\
		&\text{with entries in $\{1,\dots,n\}$}
	\end{aligned}
	\right\ra\subseteq \s_\ld V.
\]
Note that
\[
	(z)^{\wedge}|_{\s_\ld V}\left(e_{T}\cdot c_\ld\right)\in \z\left\la e_{\wt{T}}\ \middle|\ 
	\begin{aligned}
		&\wt{T}\text{ is a Young tableau of shape }\lt\\
		&\text{with entries in $\{1,\dots,n\}$}
	\end{aligned}
	\right\ra
\]
since $c_\ld$ preserves integer coefficients and $z$ has integer coefficients.
Pick a Young tableau $\wt{T}$ of shape $\lt$ with entries in $\{1,\dots,n\}$ and denote by $\wt{T}_{i,j}$ the entry in the $i$-th row and $j$-th column.
Then
\[
	\mathscr{C}_\lt(e_{\wt{T}})=e_{\wt{T}}\cdot c_\lt=\sum_{(\sigma_1,\dots,\sigma_n)}\bigotimes_{i=1}^n\left(e_{\wt{T}_{1,\sigma_1(i)}}\wedge e_{\wt{T}_{2,\sigma_2(i)}}\wedge\cdots\wedge e_{\wt{T}_{n+1-i,\,\sigma_{n+1-i}(i)}}\right)
\]
where the summation is taken over $(\sigma_1,\dots,\sigma_n)\in\mathfrak{S}_n\times\cdots\times\mathfrak{S}_1$.
In each summand, the first factor $e_{\wt{T}_{1,\sigma_1(1)}}\wedge e_{\wt{T}_{2,{\sigma_2(1)}}}\wedge \cdots \wedge e_{\wt{T}_{n,{\sigma_n(1)}}}$ equals $\pm e_1\wedge\cdots\wedge e_n$ when $\wt{T}_{1,\sigma_1(1)},\dots,\wt{T}_{n,\sigma_n(1)}$ are pairwise distinct.
Otherwise, the first factor is zero.
By fixing $\sigma_1(1),\dots,\sigma_n(1)$, the right hand side becomes zero or
\begin{equation}\label{eqn:form}
	\sum_{(\sigma'_1,\dots,\sigma'_{n-1})}(\pm e_1\wedge\cdots\wedge e_n)\otimes\bigotimes_{i=2}^n\left(e_{\wt{T}_{1,\sigma'_1(i)}}\wedge e_{\wt{T}_{2,\sigma'_2(i)}}\wedge\cdots\wedge e_{\wt{T}_{n+1-i,\,\sigma'_{n+1-i}(i)}}\right)
\end{equation}
where each $\sigma'_j$ runs over the set of bijections from $\{1,2,\dots,n+1-j\}\setminus\{1\}$ to $\{1,2,\dots,n+1-j\}\setminus\{\sigma_j(1)\}$.
Therefore, $\mathscr{C}_\lt(e_{\wt{T}})$ is a $\z$-linear combination of the elements of the above form \eqref{eqn:form}.
By taking off $e_1\wedge\cdots\wedge e_n$, we get
\begin{equation}\label{eqn:tensor_sum}
	\sum_{(\sigma'_1,\dots,\sigma'_{n-1})}\bigotimes_{i=2}^n\left(e_{\wt{T}_{1,\sigma'_1(i)}}\wedge e_{\wt{T}_{2,\sigma'_2(i)}}\wedge\cdots\wedge e_{\wt{T}_{n+1-i,\,\sigma'_{n+1-i}(i)}}\right).	
\end{equation}
Let $\mathcal{T}$ be the Young tableau of shape $\ld$ obtained by removing the box numbered $\sigma_j(1)$ in the $j$-th row of $\widetilde{T}$ for each $j$.
Then \eqref{eqn:tensor_sum} is equal to $e_{\mathcal{T}}\cdot c_{\ld}$.
Hence $\Phi^{-1}\circ \mathscr{C}_\lt(e_{\wt{T}})$ is a $\z$-linear combination of elements of the form $e_{\mathcal{T}}\cdot c_\ld$.
In summary, $(z)^{\wedge}|_{\s_\ld V}\left(e_{T}\cdot c_\ld\right)$ is a $\z$-linear combination of elements of the form $e_{\wt{T}}$, and each $\Phi^{-1}\circ \mathscr{C}_\lt(e_{\wt{T}})$ is a $\z$-linear combination of elements of the form $e_{\mathcal{T}}\cdot c_\ld$.
In particular, it holds that
\[
	\Psi(z)(A)\subseteq A.
\]
As $A$ is a finitely generated torsion free abelian group, it is a free abelian group.
Furthermore, $\s_\ld V$ has a basis
\[
	\mathcal{S}:=\{e_S\cdot c_\ld\mid S\text{ is an SSYT of shape }\ld\text{ with entries in }\{1,\dots,n\}\}
\]
whose cardinality is $2^{\binom{n}{2}}$.
Hence the rank of $A$, as a free abelian group, is at least $2^{\binom{n}{2}}$ as $\mathcal{S}\subseteq A$.
On the other hand, any $\z$-linearly independent elements of $A$ are $\q$-linearly independent, and those are again $\Bbbk$-linearly independent.
Together with $\dim \s_\ld V=2^{\binom{n}{2}}$, this implies that there are at most $2^{\binom{n}{2}}$ $\z$-linearly independent elements of $A$.
Therefore the rank of $A$ is at most $2^{\binom{n}{2}}$.
In summary, the rank of $A$ is $2^{\binom{n}{2}}$.
Pick a basis $\mathcal{B}$ of $A$ as a free abelian group, and consider it as a basis of $\s_\ld V$.
Then such $\mathcal{B}$ satisfies the claim.

Note that $\per_n\equiv \det_n$ modulo 2.
Furthermore, the map $\Psi$ is linear.
Therefore, $\det(M_{\per_n})\equiv\det(M_{\det_n})$ modulo 2.
As $\det(M_{\det_n})\equiv 1$ modulo 2, we get that $\det(M_{\per_n})$ is nonzero.
Therefore $\Psi(\per_n)$ has full rank.
Since $\Phi^{-1}$ is an isomorphism, the map $\mathscr{C}_\lt\circ(\per_n)^{\wedge}|_{\s_\ld V}$ also has full rank.
As
\[
	\mathscr{C}_\lt\circ (\per_n)^{\wedge}|_{\s_\ld V}=\gamma\mathcal{F}_{\ld,\lt}(\per_n),
\]
we get $\gamma\ne 0$ and $\mathcal{F}_{\ld,\lt}(\per_n)$ also has full rank.
Therefore, \[\rank\mathcal{F}_{\ld,{\lt}}(\per_n)=\dim \s_\ld V=2^{\binom{n}{2}}.\]

Now we show that $\rank\mathcal{F}_{\ld,{\lt}}(x_1^n)\le 2^{\binom{n-1}{2}}$.\footnote{Also cf. \cite[Lemma 3.2]{Oed16}.}
As
\[
	\mathscr{C}_\lt\circ(x_1^n)^{\wedge}|_{\s_\ld V}=\gamma\mathcal{F}_{\ld,\lt}(x_1^n)
\]
and we have seen $\gamma\ne 0$, it is enough to show
\[
	\rank \left((x_1^n)^{\wedge}|_{\s_\ld V}\right)\le 2^{\binom{n-1}{2}}.
\]
Let $S$ be an SSYT of shape $\ld$ with entries in $\{1,\dots,n\}$.
If $S$ contains $1$, then every term of $e_S\cdot c_\ld$ contains $e_1$ in some factor.
Thus $(x_1^n)^{\wedge}|_{\s_\ld V}(e_S\cdot c_\ld)=0$.
Therefore $(x_1^n)^{\wedge}|_{\s_\ld V}$ kills the elements of
\[
	\{e_S\cdot c_\ld\mid S\text{ is an SSYT of shape }\ld\text{ containing $1$ with entries in }\{1,\dots,n\}\}.
\]
Thus $\rank \left((x_1^n)^{\wedge}|_{\s_\ld V}\right)$ is at most the number of SSYTs of shape $\ld$ with entries in $\{2,\dots,n\}$, and we get $\rank \left((x_1^n)^{\wedge}|_{\s_\ld V}\right)\le 2^{\binom{n-1}{2}}$.

Therefore, it holds that
\[
	\bwrk(\per_n)\ge \frac{\rank\mathcal{F}_{\ld,{\lt}}(\per_n)}{\rank\mathcal{F}_{\ld,{\lt}}(x_1^n)}\ge \frac{2^{\binom{n}{2}}}{2^{\binom{n-1}{2}}}=2^{n-1}
\]
by \Cref{thm:Young_flat}.
Together with the upper bound $\bwrk(\per_n)\le 2^{n-1}$ given by the polarization identity \eqref{eqn:identity}, we conclude $\bwrk(\per_n)=2^{n-1}$.
\end{proof}

\section{Final remarks}\label{sec:final}
\subsection{Oeding's condition}\label{subsec:Oeding}
Oeding \cite{Oed16} proposed to consider the specific form of Young flattening $\mathcal{F}_{\ld,\mu}(f)$ for the monomial $f=x_1^{\alpha_1}x_2^{\alpha_2}\cdots x_n^{\alpha_n}$.
A partition $\nu$ is called \emph{$\alpha$-optimal} if
\[
	\nu = \left(\sum_{i=2}^n\alpha_i,\sum_{i=2}^{n-1}\alpha_i,\cdots,\sum_{i=2}^2\alpha_i\right).
\]
He set $\lambda$ to be the $\alpha$-optimal partition and $\mu=(d,\lambda)$.
Using these, he saw a possible route to prove
\[
	\bwrk(x_1^{\alpha_1}x_2^{\alpha_2}\cdots x_n^{\alpha_n})\ge \prod_{i=2}^n (\alpha_i+1)
\]
where $\alpha_1\ge \alpha_2\ge\cdots\ge\alpha_n$ and
\begin{equation}\label{eqn:Oedings_condition}
	(\alpha_2,\dots,\alpha_n)\preceq (\alpha_1,\alpha_n,\alpha_{n-1},\dots,\alpha_3).
\end{equation}
Here, $(\beta_1,\dots,\beta_l)\preceq (\gamma_1,\dots,\gamma_m)$ holds if and only if
\[
	\sum_{i=1}^j\beta_i\le \sum_{i=1}^j\gamma_i
\]
for all $j$.
If $\rank \mathcal{F}_{\lambda,(d,\lambda)}(x_1^{\alpha_1}x_2^{\alpha_2}\cdots x_n^{\alpha_n})=\dim (\s_\ld V_0)\prod_{i=2}^n(\alpha_i+1)$ holds where $V_0=\la e_2,\dots,e_n\ra$, then \Cref{conj:monomial} holds for $x_1^{\alpha_1}x_2^{\alpha_2}\cdots x_n^{\alpha_n}$.
The monomials satisfying the condition \eqref{eqn:Oedings_condition} include
\begin{itemize}
	\item $(x_1x_2\cdots x_n)^k$ for every $k$
	\item monomials of degree at most 7 except one: $x_1^2x_2^2x_3^2x_4$
\end{itemize}
which cover a large family of monomials.
Those also cover monomials with three variables, but this case is already resolved in \cite{MR4332674}.
Note that $x_1^2x_2^2x_3^2x_4$ is also resolved in \cite[Example 6.6]{MR4332674}.

\begin{question}[\cite{Oed16}]\label{question:Oeding}
	For every $\alpha=(\alpha_1,\dots,\alpha_n)$ satisfying Oeding's condition \eqref{eqn:Oedings_condition} with $\alpha$-optimal partition $\lambda$, does it hold that
	\[
		\rank\mathcal{F}_{\lambda,(d,\lambda)}(x_1^{\alpha_1}x_2^{\alpha_2}\cdots x_n^{\alpha_n})=\dim (\s_\ld V_0)\prod_{i=2}^n(\alpha_i+1)?
	\]
\end{question}

\subsection{Border tensor ranks and border Comon's conjecture}\label{subsec:border_Comon}
The assertion that
\[
	\brk(f)=\bwrk(f)
\]
holds for every symmetric tensor $f\in \Sym^d V$ is called \emph{border Comon's conjecture} \cite{MR3092255}.
Partial positive results are given in the same paper and in a paper by Ma{\'n}dziuk and Ventura \cite{MR5116873}.
On the other hand, no symmetric tensor $f$ with $\brk(f)\ne\bwrk(f)$ is known to the best of the author's knowledge.
While \Cref{thm:main} shows $\bwrk(\per_n)=2^{n-1}$, it still remains open whether $\brk(\per_n)=2^{n-1}$ holds.
It is proved that
\[
	\brk(\per_3)=4
\]
by Derksen and Makam \cite{MR3987583}.
Also, it is known that
\[
	\brk(\per_4)=8\text{ and }\brk(\per_5)\in\{15,16\}
\]
\cite{HJK25}.
If $\brk(\per_5)=15$, then $\per_5$ is a counterexample to border Comon's conjecture.
Note that $\rk(\per_5)=16$ \cite{HS26}.
The tensor rank of $\per_n$ is not yet determined for $n\ge 6$, and the border tensor rank is not yet determined for $n\ge 5$.

\subsection{Cactus ranks and the cactus barrier}\label{subsec:cactus}
The \emph{cactus rank} $\crk(f)$ of a nonzero symmetric tensor $f\in \Sym^d V$ is defined as follows.
\[
	\crk(f)=\min\left\{r\in\z\ \middle|\ 
	\begin{aligned}
		&[f]\in \la Z\ra\text{ for some $0$-dimensional}\\
		&\text{subscheme $Z\subseteq\nu_d(\p V)$ of length $r$}
	\end{aligned}
	\right\}
\]
Here, $\nu_d$ denotes the $d$-th Veronese embedding and $\la Z\ra$ denotes the linear span of $Z$ in $\p(\Sym^d V)$.
The \emph{border cactus rank} $\bcrk(f)$ of $f$ is the smallest integer $r$ such that $f$ lies in the Zariski closure of the set of $d$-forms of cactus rank at most $r$.
By definition, it holds that
\[
	\begin{array}{ccc}
		\crk(f) &\le &\wrk(f)\\[3pt]
		\rotatebox[origin=c]{90}{$\le$} & & \rotatebox[origin=c]{90}{$\le$}\\[3pt]
		\bcrk(f) &\le &\bwrk(f).
	\end{array}
\]

The \emph{cactus barrier} is a phenomenon that determinantal methods such as Young flattenings give a lower bound not only for the border Waring rank, but also for the border cactus rank \cite{MR3121848,MR3611482,MR4332674,Buc26}.
The following is a special case of a general result of Buczy{\'n}ski \cite{Buc26}.

\begin{theorem}[{cf. \cite[Theorem 1.3]{Buc26}}]\label{thm:cactus_barrier}
	Under the notation in \Cref{subsec:Young_flattening}, it holds that
	\[
		\bcrk(f)\ge \frac{\rank\mathcal{F}_{\ld,\mu}(f)}{\rank\mathcal{F}_{\ld,\mu}(x_1^d)}.
	\]
\end{theorem}

As a consequence of \Cref{thm:main} and \Cref{thm:cactus_barrier}, we get the following proposition.
\begin{proposition}
	It holds that $\bcrk(x_1x_2\cdots x_n)=2^{n-1}$ over fields of characteristic zero.
\end{proposition}
\begin{proof}
	We may assume $\Bbbk$ is algebraically closed by taking the algebraic closure.
	The proof of \Cref{thm:main} shows that
	\[
		\frac{\rank\mathcal{F}_{\ld,\lt}(x_1\cdots x_n)}{\rank\mathcal{F}_{\ld,\lt}(x_1^n)}\ge 2^{n-1},
	\]
	hence we get $\bcrk(x_1\cdots x_n)\ge 2^{n-1}$.
	The upper bound follows from the polarization identity.
\end{proof}

Note that Ranestad and Schreyer \cite{MR2842085} actually proved
\[
	\crk(x_1^{\alpha_1}x_2^{\alpha_2}\cdots x_n^{\alpha_n})=\prod_{i=2}^n(1+\alpha_i)
\]
where $\alpha_1\ge\cdots\ge \alpha_n\ge 1$, hence cactus ranks of monomials are completely determined. 
Border cactus ranks of general monomials are not.

The cactus barrier indeed causes a problem since the \emph{cactus variety}, whose underlying set consists of $d$-forms of border cactus rank at most $r$, fills up the ambient space much faster than the secant variety.
As a consequence, linear rank methods cannot give a lower bound better than the border cactus rank, no matter how hard one tries to find a good linear map.

As Buczy{\'n}ski \cite{Buc26} pointed out, only two results are known to go beyond the cactus barrier \cite{MR4634294,DM26}.
While \cite[Theorem 1.2]{BB26} implies that weak border apolarity is subject to the cactus barrier (see \cite[Section 1.5]{Buc26}), an additional criterion might be used to go beyond the cactus barrier.
See also \cite{MR4595287} for a systematic use of the fixed ideal theorem \cite{MR4332674}.

Monomials that do not satisfy Oeding's condition may require a method that goes beyond the cactus barrier.
Indeed, as Oeding pointed out in \cite[Remark 1.6]{Oed16}, Young flattenings are unlikely to reach the conjectural border Waring rank for such monomials.
Still, further study of Young flattenings of monomials seems necessary, as even determining their ranks is very hard (\Cref{question:Oeding}).

\bibliographystyle{amsalpha}
\bibliography{ref.bib}
\end{document}